\documentclass[11pt]{amsart}
\usepackage {amsmath, amssymb,   epsfig, a4wide, enumerate, psfrag}
\usepackage[utf8]{inputenc}
\usepackage[english]{babel}

\date{\today}

\author{Romain Dujardin}
\thanks{Research  partially supported by ANR project LAMBDA,  ANR-13-BS01-0002 and  a grant from the  Institut Universitaire de France}
 
\title[Saddle hyperbolicity implies hyperbolicity]{Saddle hyperbolicity implies hyperbolicity for polynomial automorphisms of $\cd$}

\address{Sorbonne Universit\'es, Laboratoire de probabilit\'es, statistique et mod\'elisation, UMR 8001,  
4 place Jussieu, 75005 Paris, France}
\email{romain.dujardin@upmc.fr}


\newcommand{\cc}{\mathbb{C}}
\newcommand{\bb}{\mathbb{B}}

\newcommand{\zz}{\mathbb{Z}}

 \newcommand{\cv}{\rightarrow}

\newcommand{\fr}{\partial}
\newcommand{\om}{\Omega}
\newcommand{\set}[1]{\left\{#1\right\}}
\newcommand{\norm}[1]{\left\Vert#1\right\Vert}
\newcommand{\abs}[1]{\left\vert#1\right\vert}
\newcommand{\cd}{{\cc^2}}

\newcommand{\rest}[1]{ \arrowvert_{#1}}

\newcommand{\unsur}[1]{\frac{1}{#1}}
\newcommand{\el}{\mathcal{L}}

\newcommand{\lrpar}[1]{\left(#1\right)}

\newcommand{\la}{\lambda}
\newcommand{\lo}{{\lambda_0}}

\newcommand{\La}{\Lambda}

\newcommand{\loc}{\mathrm{loc}}

\newcommand{\inv}{^{-1}}

\DeclareMathOperator{\supp}{Supp}

\DeclareMathOperator{\Int}{Int}

\DeclareMathOperator{\jac}{Jac}

\newtheorem{prop}{Proposition} [section]

\newtheorem{lem}[prop] {Lemma}
\newtheorem{cor}[prop]{Corollary}

\newtheorem*{mainthm}{Main Theorem}
\newtheorem{step}{Step}

\theoremstyle{remark}

\newtheorem{rmk}[prop]{Remark}

\begin{document}

\begin{abstract}
We prove that for a   polynomial diffeomorphism of $\cd$, uniform hyperbolicity on the set of saddle periodic points implies 
that saddle points are dense in the Julia set. 
In particular  $f$ satisfies Smale's  Axiom A    on $\cd$.
\end{abstract}

 \maketitle
 
 \section{Introduction}

Let $f$ be a polynomial automorphism of $\cd$ with non-trivial dynamics. For such a dynamical system
 there are two natural definitions 
for the Julia set. The first one is in terms of normal families: $J = J^+\cap J^-$ is the set of points  
at which  which neither $(f^n)_{n\geq 0}$
nor $(f^{-n})_{n\geq 0}$ is   locally equicontinuous. The second one is the closure $J^*$  of the set of saddle periodic orbits.
The inclusion $J^* \subset J$ is obvious, and  whether the reverse inclusion  holds is  one of the major open questions in 
higher dimensional holomorphic dynamics.  

Following Bedford and Smillie \cite{bs1}, we say that $f$ is {\em hyperbolic} if $J$ is a hyperbolic set for $f$.  Under this assumption  we have a rather satisfactory understanding of 
the global dynamics of $f$. Indeed it was shown in \cite{bs1} that
 the forward and backward Julia 
sets  $J^+$ and $J^-$ (see \S \ref{subs:vocabulary} below for precise definitions)
are laminated by stable and unstable manifolds, that the Fatou set is the union of finitely many cycles of attracting basins, that 
$f$ satisfies Smale's Axiom A on $\cd$ and finally that $J=J^*$. It was shown by Buzzard and Jenkins \cite{buzzard jenkins} that $f$ 
is structurally stable on $\cd$. There are also tentative   models for a description of the topological dynamics on $J$ (see Ishii 
\cite{ishii} for a survey).

On the other hand  it is sometimes more natural to postulate that $f$ is uniformly hyperbolic on $J^*$. One reason is that this 
information can be read off from the periodic points of $f$. This happens 
 for instance  in the study of the stability/bifurcation dichotomy  for families of polynomial automorphisms
 \cite{tangencies, hyperbolic}. The global consequences of hyperbolicity on $J^*$ are then 
 less easy to analyze, in particular it does not 
 {\em a priori} imply a uniform laminar structure on $J^\pm$. 

The main result of this paper is that these two notions actually coincide.

\begin{mainthm}
 Let $f$ be a  polynomial automorphism of $\cd$ with non-trivial dynamics. If $f$ is  hyperbolic on $J^*$, then 
 $J = J^*$. 
 \end{mainthm}

In particular, if $f$ is hyperbolic on $J^*$, then it is hyperbolic in the sense of \cite{bs1}. 

Recall that the Jacobian $\mathrm{Jac}(f)$ of a polynomial automorphism  is a non-zero constant: it is  dissipative when 
$\abs{\mathrm{Jac}(f)}<1$ and conservative when $\abs{\mathrm{Jac}(f)}=1$. 

This result was first announced in the dissipative case in 
\cite{fornaess}, but the published proof is not correct,
 and it has remained an intriguing  open problem since then. Recently, Guerini and Peters \cite{guerini peters} managed to establish  
the result under the more stringent assumption that $f$ is {\em substantially dissipative}, that is $\mathrm{Jac}(f) < d^{-2}$, 
where $d$ is the dynamical degree (see \S \ref{subs:vocabulary} for this notion). 
Observe that only {\em quasi-hyperbolicity} on $J^*$ 
is assumed in \cite{guerini peters} while our approach seems to  require  the full strength of hyperbolicity. 
 
The proof of the main theorem 
starts with the dissipative case (Section \ref{sec:proof dissipative}). We assume by contradiction 
that $f$ is dissipative,  hyperbolic on $J^*$ and that $J\neq J^*$. In a first stage we   show that for some $p\in J^*$, 
$J^-$  intersects $W^s(p)$ along  a non-trivial relatively open subset, which 
  is an   unexpected property  in the dissipative setting (for instance in the substantially dissipative case, the main point of \cite{guerini peters} is to show that $J^-\cap W^s(p)$ is totally disconnected). 
  The main input here is the ergodic closing lemma that we obtained in a 
  previous work \cite{closing}. 
In a second stage  we 
use the results of \cite{bs6} on the properties of stable slices of $J^-$ together with  some potential-theoretic ideas
to actually derive a contradiction. 

The  conservative case is treated in Section \ref{sec:proof conservative} by a perturbative 
argument.   If $f$ is conservative and 
 hyperbolic on $J^*$, we can find a holomorphic family $(f_\la)$ with $f_0 = f$ 
 containing dissipative parameters, on which $J^*$ moves under a holomorphic motion. 
Again we assume   that  $J^*(f)\neq J(f)$, 
and use the extension properties of the holomorphic motion of $J^*$ obtained
 in  \cite{tangencies} to derive a contradiction 
from the  previously proven dissipative case.  

\section{Preliminaries}\label{sec:prel}

In this section we recall some basic facts on the dynamics of 
polynomial automorphisms of $\cd$ and hyperbolic dynamics,  and establish a few preliminary results. 

\subsection{Vocabulary and basic facts}\label{subs:vocabulary}
Let $f$ be a polynomial diffeomorphism of $\cd$ with non-trivial dynamics. This is the case exactly when 
the {\em dynamical degree} $\lim_{n\to\infty} (\deg(f^n))^{1/n}$ is larger than 1. 
By   \cite{fm} there exists  a polynomial
change of coordinates in which $f$ is expressed as a composition of   Hénon mappings $(z,w)\mapsto (p_i(z)+a_iw,a_iz)$. 
We fix such coordinates from now on. 
The degree of $f$  is $d=\prod \deg(p_i)\geq 2$ and the relation $\deg(f^n) = d^n$ holds so that $d$ coincides with the 
dynamical degree of the original map. 

In these adapted coordinates, let $$V_R^+  = \set {(z,w)\in \cd, \ \abs{z}\geq R, \ \abs{w}< \abs{z}} \text{ and }
V_R^-  = \set {(z,w)\in \cd, \ \abs{z}\geq R, \ \abs{w}< \abs{z}}
$$
We fix  $R_0>0$   so  large   that for $R\geq R_0$ 
$f(V_R^+)\subset V_{2R}^+$ and $f\inv(V_R^-)\subset V_{2R}^-$. Hence the points of $V_R^+$ (resp. $V_R^-$) escape under forward (resp. backward) iteration. We denote by $\bb$ the bidisk 
$D(0, R_0)^2$.  The non-wandering set of $f$ is contained in $\bb$. 

An object (subset, current, or subvariety) in $\bb$ is said to be 
{\em vertical} (resp.  {\em horizontal}) if its closure in $\overline \bb$  is disjoint from 
$\set{\abs{z}=R_0}$ (resp. $\set{\abs{w}=R_0}$). A vertical subvariety has a {\em degree}, which is the number of intersection points 
with a generic horizontal line. 

Here are some  standard facts and  notation   (see e.g. \cite{bs1, bs2, bls}): 
\begin{itemize}
\item $K^\pm$ is the set of points with bounded forward orbits under $f^{\pm1}$ 
 and $K= K^+\cap K^-$. Note that $K^+$ is vertical in $\bb$ and 
 $f(\bb\cap K^+)\subset K^+$. Similarly, $K^-$ is horizontal and $f\inv(\bb\cap K^-)\subset K^-$. 
 \item The complement of $K^\pm$ is denoted by $U^\pm$. 
\item $J^\pm = \fr K^\pm$ are the forward and backward Julia sets. 
Since $f$ is dissipative, $K^- = J^-$. 
\item $J = J^+\cap J^-$ is the Julia set. 
\item $J^*\subset J$ is the closure of the set of saddle periodic points.  It is also the support of the unique 
measure of maximal entropy $\log d$. 
\end{itemize}

The {\em dynamical Green functions} $G^\pm$ are defined by $G^\pm(z,w)  = \lim d^{-n} \log^\pm\norm{f^n(z,w)}$. These 
  are non-negative  continuous   plurisubharmonic functions on $\cd$, such that $K^\pm = \set{G^\pm = 0}$ and 
  $G^\pm$ is pluriharmonic on  $U^\pm = \set{G^\pm > 0}$. 
We let  $T^\pm = dd^cG^\pm$. The maximum principle implies that    $\supp(T^\pm) = J^\pm$. 

The {\em restriction} $T^+\rest{D}$ of $T^+$ to a submanifold $D$ is locally
  given by
    $\Delta(G^+\rest{D})$ and since 
   $G^+ $ is continuous it 
    coincides\footnote{It is now standard to define $d^c   = \frac{i}{2\pi}(\overline\fr-\fr)$. Accordingly, $\Delta$ here is 
     $1/2\pi$ times the ordinary Laplacian.}  with
  the wedge product $T^+\wedge [D]$. 
A useful remark is that  
if $x$ belongs to $K^+$ and $D\subset \cd$ is a  holomorphic disk  along which  $G^+$ is harmonic,
 then $D\subset K^+$ and 
  $\lrpar{f^n\rest{\Delta}}_{n\geq 1}$ is a normal family.

 If $p$ is a saddle periodic point or   more generally if it belongs to a hyperbolic saddle set, it admits stable and unstable manifolds 
 $W^{s/u}(p)$. Each of them is an immersed Riemann surface biholomorphic to $\cc$ and by \cite{bs2, fs}
 $W^s(p)$ (resp. $W^u(p)$) is 
   dense in $J^+$ (resp. $J^-$). A key point in the present 
    paper is to analyse  the topological properties of  sets of the form $K^-\cap W^s(p)$ or $K^+\cap W^u(p)$. 
    Following \cite{tangencies}, we define the {\em intrinsic topology} 
    to be the topology induced on a stable (resp. unstable) manifold by the biholomorphism $W^s\simeq \cc$, and
    the corresponding concepts of boundary, interior, etc. will be labelled with the subscript $\mathrm{i}$: $\fr_{\,\mathrm{i}}$, 
    $\Int_{\mathrm i}$, etc. 
    
    The following basic lemma will be used several times. 

 \begin{lem}[{\cite[Lemma 5.1]{tangencies}}]\label{lem:boundary}
 Let $p$ be a saddle periodic point.
Then the boundary    of $W^s(p)\cap J^-$ relative to  the intrinsic topology in $W^s(p)$
  is contained in $J^*$.
 \end{lem}

  We denote by $W^s_\bb(p)$ the connected component of $W^s(p)\cap \bb$ containing $p$ (and accordingly for $W^u$). 
   Likewise  $W^s_\delta(p)$ is the connected component of $W^s(p)\cap B(p, \delta)$ containing $p$, and 
 $W^s_\loc(p)$ denotes an unspecified open neighborhood of $p$ in $W^s(p)$.  
   
By \cite{bls}, the currents $T^\pm$ have geometric structure, related to the decomposition of $J^\pm$ into
 stable and unstable manifolds. 
By {\em  lamination} by Riemann surfaces we mean a     closed subset
 $\el$  of  some open set $\om\subset \cd$    
such that   every $p\in \el$ admits a neighborhood $B$
biholomorphic to a bidisk, such that in the corresponding coordinates, 
a neighborhood of $p$ in $\el$  is a union of disjoint graphs (that is, a holomorphic motion) 
over the first coordinate in $B$. 
 A  positive current $S$ is {\em uniformly laminar} if there is a 
lamination of $\supp(S)$ by Riemann surfaces and in the corresponding
 local coordinates  $S$  is locally  expressed as  $\int  [\Delta_a]\, d\nu(a)$.  
 These disks will be said   {\em subordinate} to $S$. 
 
 A holomorphic disk $D$ is {\em subordinate} to $T^+$ if there exists a non-zero 
 uniformly laminar current $S\leq T^+$ such that 
 $D$ is subordinate to $S$. By \cite[Prop. 2.3]{connex}, if $p$ is any saddle point, then 
 any relatively compact disk $D\subset W^s(p)$ is subordinate to $T^+$. 

 \subsection{Stable (dis)connectivity}
It was shown in \cite{bs6} that the connectivity properties of sets of the form $K^+\cap W^u(p)$  (resp. $K^-\cap W^s(p)$)
carry deep information on the 
geometry of the Julia set. 
We say  that $f$ is   {\em stably connected} if  $U^-\cap W^s(p)$ is simply connected for some
(and then any) saddle point $p$, and {\em stably disconnected} otherwise. Equivalently, $f$ is stably disconnected if for some
 saddle point $p$, $W^s(p)\cap K^-$ admits a compact component relative to  the intrinsic topology. 
 This actually implies the stronger property that  most components of $W^s(p)\cap K^-$ are points 
  (see the proof of Lemma \ref{lem:locally simply connected} below for more details).  
 
By \cite[Cor. 7.4]{bs6},   a dissipative polynomial automorphism
 is always stably disconnected. It was observed in \cite{connex} that
this implies a strong non-extremality property for the current $T^+\rest{\bb}$: there exists   a decomposition 
 $T^+\rest{\bb} = \sum_{k=1}^\infty T_k^+$ 
where $T_k^+$ is an average of integration currents over a family 
 of disjoint vertical disks of degree $k$ (see \cite[Thm 2.4]{connex}).

 \begin{lem}\label{lem:non extremal}
Let  $f$ be dissipative and  hyperbolic on $J^*$ and let $q\in J^*$. Then $q$ belongs to the support of $T^+\rest{W^u(q)}$ and 
for $(T^+\rest{W^u(q)})$-a.e. $q'$ near $q$, $W^s_\bb(q')$ is a vertical manifold of finite
 degree in $\bb$. 
\end{lem}

\begin{proof}
The first assertion easily follows from the fact that $(f^n)_{n\geq 0}$ cannot be a normal family on $W^u_\loc(q)$  
(see \cite[Lemma 2.8]{bls}).
The second one is a consequence of 
 \cite[Thm 2.4]{connex}. Indeed as observed above  $T^+\rest{\bb}$ admits a decomposition 
 $T^+\rest{\bb} = \sum_{k=1}^\infty T_k^+$ 
where $T_k^+$ is  made of vertical disks of degree $k$. Thus  $T^+\rest{W^u_\loc(q)} 
 = T^+\wedge [W^u_\loc(q)] = \sum_k T^+_k\wedge [W^u_\loc(q)]$. Now if $\Gamma$ is a leaf of some $T_k^+$ intersecting 
 $W^u_\loc(q)$ at $q'$, then since $W^u_\loc(q)$ is subordinate to $T^-$, $q'$ belongs to $ J^*$ and 
 $\Gamma$ is a manifold through $q'$ along which forward iterates are bounded, hence 
 $\Gamma = W^s_\bb (q')$. 
 \end{proof}

\subsection{Hyperbolicity and local product structure}
Let us recall some generalities from hyperbolic dynamics, specialized to our situation.
A {\em (saddle) hyperbolic set} for $f$  is a compact invariant set 
 $\Lambda\subset \cd$ such that $T\cd\rest{\Lambda}$ admits a hyperbolic splitting, i.e. 
 $T\cd\rest{\Lambda}  = E^s\oplus E^u$, where $E^s$ and $E^u$ are continuous line bundles such that 
 $E^s$ (resp. $E^u$) is uniformly contracted (resp. expanded) by $df$. Then there exists $\delta_1>0$ such that 
 $W^s_{\delta_1}(\Lambda):=\bigcup_{p\in \Lambda} W^s_{\delta_1}(x)$  
 and $W^u_{\delta_1}(\Lambda):=\bigcup_{p\in \Lambda} W^u_{\delta_1}(x)$ form laminations in 
 the $\delta_1$-neighborhood  of $\Lambda$.  
 
 A hyperbolic set is {\em locally maximal} if 
 there exists an open neighborhood $\mathcal N$ of $\Lambda$ such that $\Lambda   = \bigcap_{n\in \zz} f^{-n}(\mathcal N)$. It 
 has {\em local product structure} if there exists $0< \delta_2 \leq \delta_1$ such that if $p, q\in \Lambda$ are such that 
 $d(p,q)<\delta_2$ then $W^s_{\delta_1}(p)\cap W^u_{\delta_1}(q)$ 
 consists of exactly  one point belonging to $\Lambda$. It turns out that 
 these two properties are equivalent 
    (see \cite[\S 4.1]{yoccoz}). 
    
We will use the following consequence  of the shadowing lemma. 

\begin{prop} \label{prop:shadowing}
Let $\Lambda$ be a compact locally maximal hyperbolic set for a polynomial diffeomorphism $f$ of $\cd$. Then there exist
positive constants $\eta$, $\alpha$ and $A$ 
such that  for every $n\geq 0$; 
if $x$ is such that $\set{x, \ldots,  f^n(x)}\subset \Lambda_\eta$ then there exists $y\in \Lambda$ such that 
$x$ is $Ae^{-\alpha n}$-close to the local stable manifold of $y$.  

A similar result holds for negative iterates: if $\set{f^{-n}(x), \ldots,  x}\subset \Lambda_\eta$ 
then there exists $z\in \Lambda$ such that 
$x$ is $Ae^{-\alpha n}$-close to the local unstable manifold of $z$.  
\end{prop}

The following corollary  is well-known.

\begin{cor} \label{cor:shadowing} If $\Lambda$ is a compact locally maximal hyperbolic set,  then
$$W^s(\Lambda) := \set{x\in \cd,  f^n(x)\underset{n\to\infty}\longrightarrow \Lambda} = \bigcup_{p\in \Lambda} W^s(p) 
\text{ and similarly } W^u(\Lambda)  = \bigcup_{p\in \Lambda} W^u(p). $$
\end{cor}

Note however that $W^s(\Lambda)$, being an increasing union of laminations,  doesn't need to have
a lamination structure (this is already false when $\Lambda$ is a hyperbolic fixed point). 

\begin{proof}[Proof (sketch)]
This is very classical. Given an orbit segment $\set{x, \ldots,  f^n(x)}$ as in the statement of the proposition, let 
$y^{(0)}$ (resp. $y^{(n)}$) be a point in $\Lambda$ such that $d(x,  y^{(0)}) <\eta$ (resp. $d(x,  y^{(n)})<\eta$). Then define a 
$\eta$-pseudo-orbit $(y^{(k)})_{k\in \zz}$ 
as follows
$$
y^{(k)}   =  \begin{cases} f^k\big(y^{(0)}\big) &\text{ for }k< 0;\\
  f^k(x)&\text{ for }0\leq k\leq n; \\  f^{k-n}\big(y^{(n)}\big)&\text{ for }k>n. \end{cases} $$ 
  Then 
  if $\eta$ is small enough by local maximality and 
  the shadowing lemma there exists a unique $y\in \La$ such that for every $k\in \zz$, $d\big(f^k(y), y^{(k)}\big)<C\eta$
(where $C$ is some constant depending on $(f, \Lambda)$, see \cite[\S 4.1]{yoccoz}). 
In particular for  $ 0\leq k\leq n$ we have   $d\big(f^k(x), f^k(y)\big) <C\eta$ and 
it follows from standard graph transform estimates that $d(x, W^s_\delta(y))\leq Ae^{-\alpha n}$. 
 \end{proof}

The next result is a simple application of the techniques of \cite{bls}.

\begin{prop}\label{prop:local_product}
If $J^*$ is hyperbolic then it has local product structure. Furthermore global stable and unstable manifolds intersect only
in $J^*$:
 $$W^s(J^*)\cap W^u(J^*)  = J^*.$$
\end{prop}

\begin{proof}
Hyperbolicity implies that for some $\delta>0$, 
if $p$ and $q$ are close enough,  $W^s_\delta(p)\cap W^u_\delta(q)$ consists of a single point $r$. 
We have to show that $r\in J^*$. Indeed, $W^s_\delta(p)$ (resp. $W^u_\delta(q)$) 
 is a disk subordinate to $T^+$ (resp. $T^-$) so there exists a non-trivial 
uniformly laminar current 
$S^+\leq T^+$  (resp. $S^-\leq T^-$) with $W^s_\delta(p)$ (resp $W^u_\delta(q)$) as a leaf. By \cite[Lem. 8.2]{bls}, $S^+$ and 
$S^-$ have continuous potentials, so the wedge product $S^+\wedge S^-$ is well defined, and geometric intersection theory of uniformly
laminar currents \cite[Lem. 8.3]{bls} implies that $r\in \supp (S^+\wedge S^-)$. Since $S^+\wedge S^-\leq T^+\wedge T^-$ 
we conclude that $r\in J^*$. 

The proof of the second assertion is similar. By local product structure, $$W^s(J^*) = \bigcup_{n\geq 0} f^{-n}(W^s_\delta(J^*)),$$ hence 
 if   $r\in W^s(J^*)$, there exists $p\in J^*$ such that $r\in W^s(p)$ so $r$ 
 belongs to a disk subordinate to $T^+$, and likewise $r\in W^u(q)$ so it belongs to a disk subordinate to $T^+$. Observe that these two disks are distinct: indeed otherwise we would have $W^s(p) = W^u(q)$ which is impossible because $W^s(p)\cap W^u(q)$ is contained in $K$ which is bounded in $\cd$. So $r$ is an isolated intersection between $W^s(p)$ and $W^u(q)$ for the leafwise topology. If this intersection is transverse, we argue as above to conclude that 
 $r$ belongs to $J^*$. If it is a tangency, then by  \cite[Lem. 6.4]{bls} for  $q'\in J^*$ close to $q$, we get   
 transverse  intersections between $W^s(p)$ and  $W^u(q')$  close to $r$ and conclude as in the transverse case. 
\end{proof}

\subsection{Stability} A theory of stability and bifurcations for polynomial automorphisms of $\cd$ was developed in 
\cite{tangencies}, centered on the notion of {\em weak $J^*$-stability}. 

A {\em branched holomorphic motion} over 
a complex manifold $\La$  in $\cd$ is a family of holomorphic graphs over $\La$ in $\La\times \cd$. It is a holomorphic
motion (i.e. an unbranched branched holomorphic motion!) 
when these graphs are disjoint. A holomorphic family $(f_\la)_{\la\in \La}$ of polynomial automorphisms of dynamical 
degree 
$d$ is {\em weakly $J^*$-stable} if the sets $J^*(f_\la)$ move under a 
branched holomorphic motion, and {\em $J^*$-stable}
 if this motion is unbranched. Note that if $f_\lo$ is uniformly hyperbolic on $J^*(f_\lo)$, then it is $J^*$-stable 
near $\lo$ in any holomorphic
 family containing $f_\lo$. 
  
  A number of properties of weakly $J^*$-stable families are established in \cite{tangencies}, including extension properties
  of the branched holomorphic motion of $J^*$ to $K$ (and more generally to $J^+\cup J^-$), that   will be used in 
  Section \ref{sec:proof conservative}. 
  These properties hold under the standing assumption that the family $(f_\la)$ is 
  {\em  substantial }\footnote{There is an unfortunate terminological conflict here: 
  this should not be confused with the notion 
  of substantial dissipativity mentioned in the introduction. }: this means that either   all members of the 
  family are dissipative, or that no  relation of a certain form between multipliers of periodic points persistently 
  holds in the parameter space $\La$. Without entering into the details, let us just note that by 
  \cite[Thm 1.4]{bhi} any open subset of the family of all polynomial automorphisms of dynamical degree $d$ is substantial. 

 \section{Proof of the main theorem: the dissipative case} \label{sec:proof dissipative}

The proof is by contradiction so assume that 
$f$ is a dissipative polynomial diffeomorphism of $\cd$, that is uniformly hyperbolic 
on $J^*$, and that $J^* \subsetneq J$. 

\begin{step}\label{step:disk}
There exists $p\in J^*$ and a holomorphic disk $\Delta \subset W^s(p)$ such that $G^-\rest{\Delta} \equiv 0$. 
\end{step}

The purpose of the  remaining steps \ref{step2} and \ref{step3} will be to show
 that such a ``queer" component of $W^s(p)\cap J^-$ actually does not exist. 

\begin{proof}
We first claim  that it is enough to show that there exists $p\in J^*$ and $q\in W^s(p)$ such that $q\in J\setminus J^*$. Indeed
observe that $W^s(p)\cap J = W^s(p)\cap J^- = W^s(p)\cap K^-$.    By Lemma \ref{lem:boundary} we have that $\fr_{\,\mathrm i}(W^s(p)\cap J^-)\subset J^*$. Hence if $q$ belongs to  $W^s(p) \cap ( J\setminus J^*)$, it belongs to the 
intrinsic interior $\mathrm{Int_i}(W^s(p) \cap J^-$, hence $G^-\equiv 0$ in a neighborhood of $q$ in $W^s(p)$.

Let now $x \in J\setminus J^*$.  By Corollary \ref{cor:shadowing}, if $\omega(x)\subset J^*$ then 
$x\in W^s(J^*)$ and  if $\alpha(x)\subset J^*$ then $x\in W^u(J^*)$. Since $W^u(J^*)\cap W^s(J^*) = J^*$, we infer that  either 
$\omega(x)\not\subset J^*$ or $\alpha(x)\not\subset J^*$. In either case we will show that both 
$W^s(J^*)\cap (J\setminus J^*)$ and $W^u(J^*)\cap (J\setminus J^*)$ are non-empty. Thus by symmetry it is enough to deal with the case where  $\omega(x)\not\subset J^*$.

Choose 
 $\eta$     so  small that  Proposition \ref{prop:shadowing} holds for $J^*$ and 
 $\omega(x)$ is not contained in $\overline{\mathcal{N}}$, where $\mathcal{N}  := (J^*)_\eta$ is the $\eta$-neighborhood of $J^*$.
 
 Consider the sequence of Cesarò averages 
 $\nu_n = \unsur{n} \sum_{k =0}^{n-1} \delta_{f^k(x)}$. By the ergodic closing lemma of \cite{closing}, 
 every cluster value of the 
 sequence $(\nu_n)$ is supported on $J^*$. It follows that the asymptotic proportion of iterates of $x$ belonging to
 $\mathcal{N}$ tends to 1, i.e. 
 $$\unsur{n}\# \set{0\leq k\leq n-1, \ f^k(x) \in \mathcal{N}} \underset{n\cv\infty}\longrightarrow 1. $$ Indeed if a positive proportion of iterates 
 stayed outside $\mathcal{N}$, any cluster limit of $\nu_n$ would have to give positive mass to $\mathcal{N}^\complement$. 
 
 We thus infer that 
 there are arbitrary long strings $\set{x_i, \ldots, x_{i+n}}$  in the orbit of $x$ that are entirely contained in $\mathcal{N}$. Indeed, if on 
 the contrary the length of such a string were uniformly bounded by some $n_0$, then the density of iterates outside 
 $\mathcal{N}$ would be bounded below by $1/(n_0+1)$. Therefore for every $n$ there exists $i_n$ such that 
  $\set{x_{i_n}, \ldots, x_{i_n+n}}\subset \mathcal{N}$. Choose $i_n$ to be minimal with this property. Since $\omega(x) \not\subset  \mathcal{N}$, there exists $j> i$ such that   $x_j\notin \mathcal{N}$. So finally for every $n$ we can find  $i_n < j_n$ such that 
  $j_n-i_n \geq n$,  $\set{x_{i_n}, \ldots, x_{j_n}}\subset \mathcal{N}$, $x_{i_n-1}\notin \mathcal{N}$ and $x_{j_n+1}\notin\mathcal{N}$. 
  
  Let $p$ (resp. $p'$) be a cluster value  of $(x_{i_n-1})$ (resp. $(x_{j_n+1})$). The points  $p$ and $p'$ belong to 
  $J$ because $x$ does, but not to $J^*$ because they lie outside $\mathcal{N}$.  
  It follows from Proposition \ref{prop:shadowing} that 
  $p\in W^s(q)$ for some $q\in J^*$ and   $p'\in W^u(q')$ for some $q'\in J^*$.  The proof is complete.
\end{proof}

\begin{rmk}\label{rmk:substantial}
If $f$ is  substantially dissipative i.e. $\jac(f)<\deg(f)^{-2}$, then the desired  contradiction readily follows from this first step. Indeed
  Wiman's theorem together with uniform  hyperbolicity 
 imply that  the vertical degree of components of stable manifolds in some large bidisk 
 $\bb$ is uniformly bounded (see \cite[Prop. 4.2]{guerini peters}
ou \cite[Lem. 5.1]{lyubich peters2}), and it follows that $J^-\cap W^s(x)$ is totally disconnected for every $x$ 
(see \cite[Thm 2.10]{connex} or \cite[Thm. 4.3]{guerini peters}), which contradicts the conclusion of Step \ref{step:disk}. 
\end{rmk}

The second and third step do not use the assumption that $J\setminus J^*\neq \emptyset$. 

\begin{step} \label{step2}
For every $p\in J^*$, 
if $\Delta$ is a component of $\mathrm{Int}_{\mathrm i}(W^s(p)\cap J^-)$, then $\Delta$ is unbounded for the leafwise topology. 
\end{step}

\begin{proof}
Note first that by the maximum principle, any component of $\mathrm{Int}_{\mathrm i}(W^s(p)\cap J^-) =
\mathrm{Int}_{\mathrm i}(W^s(p)\cap K^-) $ is simply connected, so
$\Delta$ is a topological disk.  Assume by contradiction that $\Delta$ is bounded for the leafwise topology. 
Then iterating forward a few times if needed, we can 
suppose that $\Delta$ is entirely contained in a local product structure box. 

More precisely for  small $\delta>0$, we can fix  
 holomorphic local coordinates $(z,w)$ near $p$  in which  $p=(0,0)$, $W^s_{2\delta}(p) = \set{z=0}$ and 
$W^u_{2\delta}(p) = \set{w=0}$, and assume
$\overline{\Delta}$ is contained in $W^s_{\delta}(p)$. Note that by Lemma \ref{lem:boundary}, $\fr_{\,\mathrm i}\Delta \subset J^*$. 
We can assume that for every $q \in  W^s_{\delta}(p)\cap J^*$, 
 $W^u_{2\delta}(q)$  contains a graph over the disk $D(0, \delta)$ in the first coordinate, with slope bounded by $1/2$. 
Then if $\abs{z_0} \leq \delta$, the holonomy $h_{0, z_0}^u$  
along local unstable leaves is well defined on $ W^s_\delta(p)\cap J^*$ and maps $ W^s_\delta(p)\cap J^*$ into 
$\set{z=z_0}\cap J^-$. This holonomy is a holomorphic motion so   
by Slodkowski's theorem \cite{slodkowski} it extends to a holomorphic motion of $ W^s_\delta(p)$.  In particular the motion 
of $\fr_{\,\mathrm i} \Delta$ extends to a motion of $\Delta$ and it   
makes sense to speak about
  $h_{0, z_0}^u(\Delta)$. This is an open subset of $\set{z=z_0}$, 
  which is topologically a disk and whose boundary is contained 
  in $J^-$. Thus for every $n\geq0$, $f^{-n}(\fr(  h_{0, z_0}^u(\Delta)))$ in contained in $\bb$ 
  and by the maximum principle, the same holds for $f^{-n}(   h_{0, z_0}^u(\Delta))$.
  
   Finally, 
  $U : =\bigcup_{\abs{z_0}<\delta} h_{0, z_0}^u(\Delta)$ is an open set whose negative iterates remain in $\bb$, hence it 
  is contained in   the Fatou set of $f^{-1}$. But since $f$ is dissipative, this Fatou set  is empty, 
   which is the desired contradiction. 
\end{proof}

\begin{step} \label{step3}
 The unstable holonomy preserves the decomposition 
 $$W^s(p) = (W^s(p)\cap J^-)\sqcup(W^s(p)\cap U^-).$$ 
\end{step}

To make this statement precise, observe that   for every $p\in J^*$, 
the components of the complement of 
$\fr_{\,\mathrm i}(W^s(p)\cap J^-)$ in  $W^s(p)$ can be divided into two types: 
components of $\mathrm{Int}_{\mathrm i}(W^s(p)\cap J^-)$ and components of 
$W^s(p)\cap U^-$ (note that since $U^-$ is open in $\cd$, $ W^s(p)\cap U^-$ is open for the intrinsic topology as well). 
Consider as above local coordinates $(z,w)$ 
near $p$  in which  $p=(0,0)$, $W^s_{2\delta}(p) = \set{z=0}$ and 
$W^u_{2\delta}(p) = \set{w=0}$.
The unstable holonomy $h^u_{0, z_0}$ is initially only defined for points of  
$W^s_{\delta}(p)\cap J^* = \fr_{\,\mathrm i}(W_{\delta}^s(p)\cap J^-)$, however by S{\l}odkowski's theorem 
it can be  extended to $W^s_{\delta}(p)$.
By Step \ref{step2}, components of  $\mathrm{Int}_{\mathrm i}(W^s(p)\cap J^-)$
 are leafwise unbounded so they cannot be contained in $W^s_{2\delta}(p)$. Obviously, 
the same holds for components of $W^s(p)\cap U^-$. 

If $q$ belongs to $J^* \cap W^u_{\delta}(p)$, the extended holonomy  $h^u_{p,q}$    
defines  a  homeomorphism  $W^s_{\delta}(p) \to h^u_{p,q}(  W^s_{\delta}(p))$. By local product structure this
homeomorphism preserves $J^*$
so any component of  
$W^s_{\delta}(p) \setminus J^*$ is mapped onto a component of 
$h^u_{p,q}(  W^s_{\delta}(p))\setminus J^*)$, which is itself contained in a component of 
$W^s(q)\setminus J^*$. 
The  claim of Step \ref{step3} is that  the 
extended holonomy  $h^u_{p,q}$ preserves the type of 
components.

Since it doesn't make sense to transport a whole  leafwise unbounded component   by unstable holonomy, to prove this assertion
 we need to find a  criterion that recognizes
 the type of a component  just from local topological properties near a point of its boundary. As already said 
 the maximum principle implies that any component of  $\mathrm{Int}_{\mathrm i}(W^s(p)\cap J^-)$ is simply connected. 
 Thus Step \ref{step3} follows from: 
 
 \begin{lem}\label{lem:locally simply connected}
 If $\om$ is a component of $W^s(p)\cap U^-$, then $\om$ is not 
 simply connected near any   point of $\fr \om$, more precisely: if $q\in \fr\om$ and 
 $N$ is any neighborhood of $q$, there is a loop in $N\cap \om$, homotopic to a point in $N$, and
  enclosing a component of $W^s(p)\cap J^-$
 \end{lem}

\begin{proof} 
Since $f$ is dissipative by \cite[Cor. 7.4]{bs6} it is stably disconnected. 
It follows that almost every unstable component  of $K^+$
is a point (see \cite[Thm 7.1]{bs6} and also \cite[Thm 2.10]{connex}).  
More specifically, if $\mu$ is the unique measure of maximal entropy, then for $\mu$-a.e. $x$, the measure $T^-\rest{W^s(x)}$ (which is locally given by the wedge product 
$T^-\wedge[W^s(x)]$) gives full mass to the point components 
of  $J^-\cap W^s(x)$.  Obviously by Lemma \ref{lem:boundary}
every such point component belongs to $J^*$ so we can transport it to nearby stable manifolds by unstable 
holonomy. In addition, the measure $T^-\rest{W^s(x)}$ is holonomy invariant (see \cite[Thm 6.5]{bs1} or 
\cite[Thm 4.5]{bls}) so if $x$ is such that $T^-\wedge[W^s_\delta(x)]$ gives full mass to point components, then the 
same holds for nearby $x'$. Thus we conclude that this property holds for {\em every} $p\in J^*$: $T^-\rest{W^s(p)}$ gives full mass
to the point components of $J^-\cap W^s(p)$.

\begin{lem} \label{lem:harmonic measure}
Let $p\in J^*$ and 
$\om$ be a component of $W^s(p)\cap U^-$  
such that $\om$ is locally simply connected near some $q\in \fr\om$. 
Then $\fr \om$ has positive $(T^-\rest{W^s(p)})$-measure. 
\end{lem}

This proves Lemma \ref{lem:locally simply connected}. Indeed, 
 assuming that 
  $\om$ is locally simply connected near some  $q\in \fr\om$, Lemma \ref{lem:harmonic measure} asserts  that 
    $T^-\rest{W^s(p)}$ carries positive mass on a non-trivial 
 continuum so $f$ cannot be  stably disconnected. On the other hand   $f$ must be stably disconnected 
 because it is dissipative, 
 and we reach a contradiction.
 \end{proof}

The idea  of Lemma \ref{lem:harmonic measure} is as follows: 
every neighborhood of $q$ in $\fr\om$ has positive harmonic measure when viewed 
from $\om$.  But the harmonic measure viewed from $\om$ is absolutely continuous with 
respect to $T^-\wedge[W^s(p)]$, hence the result. 
The formalization of this argument requires some elementary potential theory, for which we refer the reader to Doob's classical monograph\footnote{Note that Doob works with superharmonic functions so all inequalities have to be reversed.} \cite{doob}. 

In particular we shall 
use the formalism of {\em sweeping} (or {\em balayage}). Let $D$ be a smoothly bounded domain
in $\cc$, $A$ a non-polar compact subset of $D$ and $\nu$ a positive measure on $D$. 
The swept measure $\rho_{\nu, D, A}$ 
of   $\nu$  on $A$    is the distribution on $A$  of the exit point of the Brownian motion in $D\setminus A$ 
 whose starting point is distributed according to 
    $\nu$. In particular its mass is lower than that of $\nu$ since a positive proportion of Brownian 
 paths escape from $\fr D$. If $G_{\nu, D}$ is the Green potential of $\nu$ in $D$, that is the unique negative subharmonic 
 function on $D$ such that $G_{\nu, D}\rest{\fr D} = 0$ and $\Delta G_{\nu, D} = \nu$, then the swept measure of $\nu$ on $A$ is 
 $\Delta R_{\nu, D, A}$, where 
 \begin{equation}\label{eq:reduction}
 R_{\nu, D, A}  (z)  = \sup\set{u(z) , \ u   \leq 0 \text{ subharmonic on } D \text{ and }u\leq  G_{\nu, D}\text{ on } A}
 \end{equation}
(see Sections 1.III.4, 1.X and 2.IX.14 in \cite{doob}). If  $\nu$ and $\nu'$ have their supports disjoint from $A$, then the 
corresponding swept measures are mutually absolutely continuous (as follows from instance from Theorem 1.X.2 
in \cite{doob}).

\begin{proof}
We first claim that we can  shift $q$ slightly so that the assumptions of the lemma hold  and in addition
$W^s_\bb(q')$ is of bounded vertical degree. Indeed $q$ belongs to $J^*$
and for $q'\in W^u_\delta (q)\cap J^*$, there is a component of $W^s_\delta(q')\setminus J^*$
 corresponding to $\om$ under unstable holonomy, which is locally simply connected near $q'$. 
 Since $G^-$ is continuous, if $q'$ is close enough to $q$, it takes positive values on that component, so we 
 infer that the property that $\om$ is a component of $U^-$ is open. Now by Lemma \ref{lem:non extremal},
  for $(T^+\rest{W^u_\delta})$-a.e. $q'$, $W^s_\bb(q')$  is a vertical manifold in $\bb$ of finite degree 
 which establishes our claim. Without loss of generality rename $q'$ into $q$. 
  For every $g_0 < \min_{\fr\bb} G^-$, the component of $\set{G^-<g_0}$ containing $q$ in $W^s(q)$ is relatively compact for the 
  intrinsic topology. 
  We fix such a $g_0$ which is not a critical value of $G^-$  and let $D$ be the corresponding component, which is a 
  smoothly bounded topological disk. 
   From now on we work exclusively in $D$.
 
 By assumption there is a   neighborhood $N$ of $q$ in $D$ and a component $U$ of $\set{G^->0}\cap N$ that is 
 simply connected.  We have to show that $\fr U\cap N$ has positive mass relative to $\Delta G^-$.  
 We choose $N$ to be closed so that $\fr U\cap N$ is compact.
 First, observe that for every $z_0 \in U\cap N$, the probability that 
 the  Brownian motion  issued from $z_0$ hits $\fr U\cap N$ before leaving 
$U$ is positive.  Therefore the swept measure $\rho_{\delta_{z_0}, D, \fr U\cap N}$ has positive mass and 
to prove the lemma it is enough to show that   is absolutely continuous with respect to   $\Delta G^-$. Recall that the measure class
of the swept measure does not depend on the starting point so we can replace $\delta_{z_0}$ by an arbitrary positive 
measure on $D \setminus (D\cap K^-)$. Let $0<g_1< g_0$ and $\mu_{g_1} := \Delta (\max(G^-, g_1))$ be the natural measure 
induced by $G^-$ on the level set $\set{G^-  = g_1}$. We choose $\mu_{g_1}$ for the initial distribution of Brownian motion.  Since $G^-\equiv g_0$ on $\fr D$, 
the Green function $G_{\Delta G^-, D}$ of  the restriction of 
$\Delta G^-$ to $D$  is equal to  $G^--g_0$, and likewise
$$G_{\mu_{g_1}, D} = \max(G^-, g_1) - g_0.$$ Thus from \eqref{eq:reduction} we get that 
\begin{align*} 
R_{\mu_{g_1} , D, K^-\cap D }&= \sup\set{u(z) , \ u \text{ s.h.} \leq 0 \text{ on } D \text{ and } u\leq G_{\mu_{g_1}, D} 
 \text{ on } K^-\cap D }\\
 &= \sup\set{u(z) , \ u \text{ s.h.} \leq 0 \text{ on } D \text{ and } u\leq  g_1-g_0 
 \text{ on } K^-\cap D }\\
  &=  \abs{g_1-g_0}\frac{G^--g_0}{g_0}
  \end{align*}
and finally $$\rho_{\mu_{g_1}, D, K^-\cap D} = \frac{g_0-g_1}{g_0} \Delta G^-.$$
The proof is complete.
\end{proof}

 \begin{step} \label{step4}
 Conclusion.
\end{step}

We just have to assemble the three previous steps. Assume as before by contradiction 
 that  $f$ is dissipative, uniformly hyperbolic on $J^*$ and $J^*\subsetneq J$. Then by Step \ref{step:disk} there exists $p\in J^*$ and
 a ``queer'' component $\om$ of $W^s(p)\setminus J^*$ along which $G^-\equiv 0$. Pick  $q\in \fr \om$. 
 By Lemma \ref{lem:boundary},
 $q\in J^*$ so we can follow $\om \cap W^s_\delta(q)$ using 
 the holonomy along local unstable manifolds. Then for $q'\in W^u_\loc(q)$ near $q$, the holonomy image 
 $h_{q,q'}(\om \cap W^s_\delta(q))$ is contained in a queer component of $W^s(q')\setminus J^*$, 
  which must be leafwise unbounded by Step \ref{step2}. 
 On the other hand by Lemma 
 \ref{lem:non extremal}, for generic $q'$ in   $W^u_\loc(q)$ (relative to the transverse measure $T^+\rest{W^u_\loc(q)}$) 
 $W^s_{\bb}(q')$ is of bounded 
 degree, in particular any component of $K^-\cap W^s_{\bb}(q')$ is leafwise bounded. This contradiction finishes the proof. \qed

 \section{Proof of the main theorem: the conservative case}\label{sec:proof conservative}
 
 Again the proof is by contradiction, so assume that $f$ is a conservative 
 polynomial automorphism of $\cd$ such that $J^*$ is a hyperbolic set and $J^*\subsetneq J$. 
 We will use  a perturbative argument and the  dissipative case of the theorem to reach a contradiction. 
 
 Assume that $f$ is written as a product of Hénon mappings $f = h_1\circ \cdots \circ h_k$ 
  and let $(f_\la)_{\la\in B}$ be a parameterization 
 of a neighborhood of $f$    in the space of such products, that is the space of coefficients of the 
 $h_i$, and such that $f_0 = f$. We can assume that $B$ is a ball in $\cc^N$ for some $N$.  
  Since $\la\mapsto \mathrm{Jac}(f_\la)$ is an open map, there 
 exist parameters arbitrary close to  0 for which $f_\la$ is dissipative. As already said, by \cite[Thm 1.4]{bhi} there 
 is no persistent relation between multipliers of periodic orbits so the family is substantial
 in the sense of \cite{tangencies}. 
 
 Since $f_0$ is hyperbolic on $J^*(f_0)$ the family $(f_\la)$ is $J^*$-stable in a neighborhood of the origin, that is, 
 $J^*(f_\la)$ moves under a 
 holomorphic motion. Reducing the parameter space we can assume that $J^*$ is hyperbolic throughout $B$.
 Pick a point $p = p(0) \in J(f_0)\setminus J^*(f_0)$. 
 It was shown in \cite[Thm 5.12]{tangencies} that in a (weakly) $J^*$-stable family, 
 the  motion of $J^*$ extends to a branched 
 holomorphic motion of  $K$. Thus there exists a holomorphic 
 continuation $p(\la)$ of $p(0)$ such that for every $\la\in B$, $p(\la)$ belongs to $K(f_\la)$. Furthermore  
  for every 
 $\la\in B$, $p(\la)$ is disjoint from $J^*(f_\la)$. Indeed if for  some $\la_0\in B$ we had  $p(\la_0) 
 \in J^*(f_{\la_0})$, then by \cite[Lem 4.10]{tangencies}
 $p(\la)$ would have to coincide  throughout the family $(f_\la)$ 
 with the 
 natural continuation of $p(\la_0)$ as a point of the hyperbolic set  $J^*$, which is not the case since $p(0)\notin J^*(f_0)$.
 
 Let now $\la_1\in B$ be such that $f_{\la_1}$ is dissipative. Then  
  by the first part of the proof
 $J(f_{\la_1}) = J^*(f_{\la_1})$, and  $K(f_{\la_1})\setminus J(f_{\la_1})$ is non-empty since it contains $p(\la_1)$. For a 
 dissipative hyperbolic map $$K\setminus J=(K^+\cap J^-)\setminus (J^+\cap J^-) = \mathrm{Int}(K^+)\cap J^-,$$ so we 
 deduce that $\mathrm{Int}(K^+(f_{\la_1}))$ is non-empty.  By \cite{bs1},  $\mathrm{Int}(K^+(f_{\la_1}))$ is a 
 finite  union of 
 attracting basins of periodic sinks, therefore $f_{\la_1}$ admits an attracting periodic point. 
 On the other hand by \cite[Thm 4.2]{tangencies}, periodic points stay of constant type in a $J^*$-stable family (this 
 holds even in the presence of conservative maps, provided the family is substantial), so $f_0$
must have an attracting orbit, which is contradictory since it is conservative.  The proof is complete. \qed


\begin{thebibliography}{[ABCD]}	
\bibitem[BD]{hyperbolic} Berger, Pierre; Dujardin, Romain. {\em On stability and hyperbolicity 
for polynomial automorphisms of $\mathbb{C}^2$.} Ann. Sci. Ec. Norm. Sup. (4) 50 (2017), 449--477.
   \bibitem[BS1]{bs1} Bedford, Eric; Smillie, John. 
  {\em Polynomial
diffeomorphisms of $\cc^ 2$: currents, equilibrium measure and
hyperbolicity.} 
  Invent. Math. 103 (1991), 69--99. 
 \bibitem[BS2]{bs2} Bedford, Eric;  Smillie, John. {\em
Polynomial diffeomorphisms
of $\cc^2$. II: Stable manifolds and recurrence.}
J. Amer. Math. Soc. 4 (1991), 657--679. 
\bibitem[BS3]{bs3} Bedford, Eric; Smillie, John. {\em Polynomial
diffeomorphisms of $\cc^ 2$. III. Ergodicity, exponents and entropy of
the equilibrium measure.} Math. Ann. 294 (1992),  395--420.
 \bibitem[BLS]{bls}
Bedford, Eric;  Lyubich, Mikhail;  Smillie, John. {\em Polynomial
diffeomorphisms of $\cc^ 2$. IV. The measure of maximal entropy and
laminar currents. }Invent. Math.  112  (1993), 77--125.
\bibitem[BS6]{bs6} Bedford, Eric; Smillie, John.
{\em Polynomial diffeomorphisms of $\cd$. VI. Connectivity of  $J$.}
  Ann. of Math. (2)  148  (1998),   695--735.
  \bibitem[BHI]{bhi}   Buzzard, Gregery T.; Hruska, Suzanne Lynch; Ilyashenko, Yulij.
{\em Kupka-Smale theorem for polynomial automorphisms of $\mathbb{C}^2$ and 
persistence of heteroclinic intersections.}  Invent. Math.  161  (2005),  45--89.
\bibitem[BJ]{buzzard jenkins}  Buzzard, Gregery T.; Jenkins, Adrian.{\em
 Holomorphic motions and structural stability for polynomial automorphisms of $\mathbb{C}^2$.}
  Indiana Univ. Math. J. 57 (2008),   277--308.
\bibitem[Do]{doob}  Doob, Joseph L. {\em 
 Classical potential theory and its probabilistic counterpart.}
  Reprint of the 1984 edition. Classics in Mathematics. Springer-Verlag, Berlin, 2001. xxvi+846 pp. 
 \bibitem[Du1]{connex}  Dujardin, Romain.
 {\em Some remarks on the connectivity of Julia sets for
  2-dimensional diffeomorphisms,}
  in {\em  Complex Dynamics,} 63-84, Contemp. Math., 396, 
Amer. Math. Soc., Providence, RI, 2006.
\bibitem[Du2]{closing} Dujardin, Romain. {\em A closing lemma for polynomial automorphisms of $\cd$.} Preprint 
{\tt arxiv:1709.01378}.  
 \bibitem[DL]{tangencies} Dujardin, Romain; Lyubich, Mikhail. 
 {\em Stability and bifurcations for dissipative polynomial automorphisms of $\cd$}. 
Invent. Math. 200 (2015), 439--511. 
\bibitem[F]{fornaess} Forn\ae ss, John Erik. {\em  The Julia set of Hénon maps.}
 Math. Ann. 334 (2006),   457--464.
\bibitem[FS]{fs} Forn\ae ss, John Erik; 
Sibony, Nessim {\em Complex H{\'e}non mappings in $\cd$ and Fatou-Bieberbach  
domains.}  Duke Math. J.  65  (1992), 345--380.
\bibitem[FM]{fm} Friedland, Shmuel; Milnor, John.
\newblock{\em Dynamical properties of plane polynomial automorphisms.}
Ergodic Theory Dynam. Systems 9 (1989),  67--99.
 \bibitem[GP]{guerini peters} Guerini, Lorenzo;   Peters, Han. {\em Julia sets of complex Hénon maps.} Preprint 
  {\tt arxiv:1706.00220}. Internat. J. Math. to appear.
 \bibitem[I]{ishii} Ishii, Yutaka. {\em Dynamics of polynomial diffeomorphisms of $\mathbb{C}^2$: combinatorial and topological aspects.} Arnold Math. J. 3 (2017),  119--173.
 \bibitem[LP]{lyubich peters2}  Lyubich, Mikhail; Peters, Han. 
{\em Structure of partially hyperbolic  Hénon maps.}  Preprint {\tt arxiv:1712.05823}.
  \bibitem[S]{slodkowski} S{\l}odkowski, Zbigniew. {\em Holomorphic
      motions and polynomial hulls.} Proc. Amer. Math. Soc.  111
    (1991), 347--355.
  \bibitem[Y]{yoccoz} Yoccoz, Jean-Christophe. {\em Introduction to hyperbolic dynamics.} Real and complex dynamical systems (Hillerød, 1993), 265–291, NATO Adv. Sci. Inst. Ser. C Math. Phys. Sci., 464, Kluwer Acad. Publ., Dordrecht, 1995.
    \end{thebibliography}
 \end{document}